\documentclass[11pt,reqno]{amsart}
\usepackage{amsthm, amsfonts, amssymb, color}
 \usepackage{mathrsfs}
\usepackage{enumerate}
\usepackage{amsmath}
 \usepackage{amstext, amsxtra}
  \usepackage{txfonts}
 \usepackage[colorlinks, linkcolor=black, citecolor=blue, pagebackref, hypertexnames=false]{hyperref}
\usepackage{graphicx}
\usepackage[symbol]{footmisc}
\usepackage{cite}
 \allowdisplaybreaks
\newtheorem{theorem}{Theorem}[section]
\newtheorem{proposition}[theorem]{Proposition}
\newtheorem{corollary}[theorem]{Corollary}
\newtheorem{lemma}[theorem]{Lemma}

\newtheorem{remark}[theorem]{Remark}

\numberwithin{equation}{section}
\theoremstyle{definition}

\title[]
 {Unique continuation through hyperplane for higher order parabolic and Schr\"{o}dinger equations}

\author{Tianxiao Huang}

\address{Tianxiao Huang, School of Mathematics (Zhuhai), Sun Yat-sen University, Zhuhai, Guangdong 519082, China}
\email{htx5@mail.sysu.edu.cn}



\begin{document}

\begin{abstract}
Consider the higher order parabolic operator $\partial_t+(-\Delta_x)^m$ and the higher order Schr\"{o}dinger operator $i^{-1}\partial_t+(-\Delta_x)^m$ in $X=\{(t,x)\in\mathbb{R}^{1+n};~|t|<A,|x_n|<B\}$, where $m$ and $n$ are any positive integers. Under certain lower order and regularity assumptions, we prove that if solutions to the linear problems vanish when $x_n>0$, then the solutions vanish in $X$. Such results are global if $n>1$, and we also prove some relevant local results.
\end{abstract}

\maketitle

%\tableofcontents

%%%%%%%%%%%%%%%%%%%%%%%%%%%%%%%%%%%%%%%%%%%%%%%%%%%%%%%%%%%%%%%%%%%%%%%%%%%%%%%%%%

\section{Introduction}

Let $m$, $n$ be any positive integers. Consider the higher order parabolic operator
\begin{equation*}
P_p(D_t,D_x)=\partial_t+(-\Delta_x)^m
\end{equation*} 
and the higher order Schr\"{o}dinger operator
\begin{equation*}
P_s(D_t,D_x)=D_t+(-\Delta_x)^m
\end{equation*}
where $(t,x)\in\mathbb{R}\times\mathbb{R}^n$, $(D_t,D_x)=i^{-1}(\partial_t,\partial_x)$ and $\Delta_x=\partial_{x_1}^2+\cdots+\partial_{x_n}^2$ is the spatial Laplacian. This paper studies the unique continuation properties of these operators through some hypersurfaces. Our main results are the following two theorems considering a  hyperplane as the spatial boundary:

\begin{theorem}\label{thm11}
	Let $X=\{(t,x)\in\mathbb{R}^{1+n};~|t|<A,|x_n|<B\}$ for some $A,B>0$. Suppose $\partial_x^\alpha u\in L^2(X)$ for $|\alpha|<2m$, $P_p(D_t,D_x)u\in L^2(X)$, and in $X$ that
	\begin{equation}\label{eq11}
	|P_p(D_t,D_x)u|\leq C\sum_{|\alpha|\leq[\frac{3m}{2}]}|\partial_x^\alpha u|.
	\end{equation}
	If $u\equiv0$ when $x_n>0$, then $u\equiv0$ in $X$.
\end{theorem}

\begin{theorem}\label{thm12}
	With the same $X$, $A$ and $B$ above, suppose $\partial_tu,\partial_x^\alpha u\in L^2(X)$ for $|\alpha|\leq2m$, and in $X$ that
	\begin{equation}\label{eq12}
	|P_s(D_t,D_x)u|\leq C
	\begin{cases}
	\sum_{|\alpha|\leq[\frac{3m}{2}]}|\partial_x^\alpha u|\quad &if~m\geq3,\\
	\sum_{|\alpha|\leq2m-2}|\partial_x^\alpha u|+\sum_{|\alpha|=2m-2}|\partial_x^{\alpha+e_n}u|\quad &if~m=1,2,
	\end{cases}
	\end{equation}
	where $e_n$ is the n-th spatial unit vector. If $u\equiv0$ when $x_n>0$, then $u\equiv0$ in $X$.
\end{theorem}

If we look at the simplest case considering solutions to $P_p(D_t,D_x)u=0$ and to $P_s(D_t,D_x)u=0$, the claims in these theorems are actually implied \textit{locally} by the well-known Holmgren's uniqueness theorem (see \cite[Theorem 8.6.5]{H1}), because the principal symbols of $P_p(D_t,D_x)$ and of $P_s(D_t,D_x)$ are both $p(\eta,\xi)=|\xi|^{2m}$ ($(\eta,\xi)\in\mathbb{R}^{1+n}$), and the hyperplane $\{x_n=0\}$ is non-characteristic everywhere for these operators. More generally, if we consider any $C^1$ \textit{spatial boundary} $S$ in $\mathbb{R}^{1+n}$, whose normals are all orthogonal to the time axis making $S$ the most typical type of non-characteristic hypersurfaces for evolution operators like $P_p(D_t,D_x)$ and $P_p(D_t,D_x)$, the Holmgren's uniqueness theorem also implies local unique continuation property across $S$ for solutions to
\begin{equation}\label{e13}
P_p(D_t,D_x)u=\sum_{|\alpha|<2m}A_\alpha(t,x)D_x^\alpha u\quad\text{and~to}\quad P_s(D_t,D_x)u=\sum_{|\alpha|<2m}A_\alpha(t,x)D_x^\alpha u,
\end{equation}
where the lower order coefficients $A_\alpha(t,x)$ are analytic. It is then natural to ask what happens when $A_\alpha(t,x)$ are not analytic.

In spatial dimension $n=1$, where $S$ is locally only a segment parallel to the time axis, our Theorem \ref{thm11} and Theorem \ref{thm12} are local results to this problem. Actually, Isakov \cite{I} had earlier proved that when $A_\alpha(t,x)$ are locally bounded, local unique continuation holds for $H_\mathrm{loc}^{2m}$ solutions to the equations \eqref{e13} across $S$, and this is somehow a better result than ours in the sense that lower order terms up to order $2m-1$ are all allowed.

In higher spatial dimension $n>1$, this problem has been more explored in the second order case $m=1$. In \cite{I}, Isakov considered the parabolic operators
\begin{equation}\label{e14}
\partial_t-\sum_{j,k}a_{jk}(t,x)\partial_{x_j}\partial_{x_k}+b(t,x)\cdot\nabla_x+c(t,x),
\end{equation}
where $\sum_{j,k}a_{jk}(t,x)\partial_{x_j}\partial_{x_k}$ is uniformly elliptic with $C^1$ coefficients and all lower order coefficients are locally bounded, improving an earlier result of Saut and Scheurer \cite{SS}. More precisely, suppose $S$ is a $C^2$ spatial boundary defined by $\psi(x)=0$ for $(t,x)$ in a neighborhood of $0\in\mathbb{R}^{1+n}$, $0\in S$, $x_n>0$ for points $(t,x)\in S$ except those in the form $(t,0)$, and suppose $\psi(x)>0$ implies $x_n>0$. An essential result in \cite{I} is that if $u\in H_\mathrm{loc}^2$ is a solution to 
\begin{equation}\label{e15}
\partial_tu-\sum_{j,k}a_{jk}(t,x)\partial_{x_j}\partial_{x_k}u+b(t,x)\cdot\nabla_xu+c(t,x)u=0
\end{equation} 
in a neighborhood of $0$, and if $u\equiv0$ when $\psi(x)>0$, then $u\equiv0$ in some neighborhood of $0$. Since the initial vanishing of $u$ is very tiny, the local unique continuation property through spatial boundaries for parabolic operators \eqref{e14} has few essential difference to the case of analytic coefficients, and Isakov's result also imlies our Theorem \ref{thm11} when $m=1$. Isakov \cite{I} also considered the Schr\"{o}dinger equation
\begin{equation}\label{e16}
i\partial_tu-\Delta_xu+b(t,x)\cdot\nabla_xu+c(t,x)u=0
\end{equation}
with locally bounded coefficients, however, the result is much weaker than the parabolic case. It could only be shown that if $u\equiv0$ when $\psi(x)<0$, then $u\equiv0$ in some neighborhood of $0$. Geometrically speaking, the initial vanishing here is assumed to be on a place "\textit{strictly larger}" than the half space locally. 

The above second order results in the case of non-analytic coefficients were all proved by establishing certain Carleman estimates, and such a way has become the main stream in the study of unique continuation problems since Carleman first introduced it in \cite{C}. In the higher order case $m>1$, however, it is in general difficult to establish Carleman estimates for higher order operators like $P_p(D_t,D_x)$ and $P_s(D_t,D_x)$ whose Fourier symbols have zeros with high multiplicities, and unique continuation through non-characteristic hypersurfaces in higher spatial dimensions has not been much studied. We will prove Theorem \ref{thm11} and Theorem \ref{thm12} by establishing appropriate Carleman estimates, but notice that when $n>1$, our Theorem \ref{thm11} and Theorem \ref{thm12} are only global in the $x'=(x_1,\cdots,x_{n-1})$ variables, to better explain why we have such consideration, we next shortly discuss how classical ideas of proving Carleman estimates may fail when studying $P_p(D_t,D_x)$ and $P_s(D_t,D_x)$.

In the general theory of local unique continuation through a hypersurface $S=\{\psi(x)=\psi(x_0)\}$ for a differential operator $P(x,D)$, establishing the following type of Carleman estimate is crucial:
\begin{equation}\label{eq14}
\sum_{|\alpha|<\mathrm{deg}\,P}\tau^{2(\mathrm{deg}\,P-\delta-|\alpha|)}\int e^{2\tau\phi}|D^\alpha f|^2\leq C\int e^{2\tau\phi}|p(\epsilon x,D)f|^2.
\end{equation}
Here $p(x,D)$ is the principal part of $P(x,D)$, $\tau$ is a large parameter which indicates uniqueness when it goes to $+\infty$, $\epsilon$ is a small but fixed parameter which controls some errors from coefficients, $f$ is any smooth function supported in a sufficiently small neighborhood of $x_0$, and $\phi$ is a suitable weight function with a local level set lying strictly on the side of $S$ where the solution vanishes, except $x_0$ where $\nabla\phi$ points at the vanishing side. The most important parameter $\delta\geq0$ in \eqref{eq14} characterizes the \textit{loss of derivatives} in the sense that $\tau$ and $D$ have the same strength from the view of semi-classical calculus.

In most cases where \eqref{eq14} can be proved, it seems that the loss of derivatives $\delta$ cannot exceed $1$. The general study that explains this situation to the most extent may be the historical Calder\'{o}n's uniqueness theorem which, in a typical way (e.g. see H\"{o}rmander \cite[Chapter 28]{H4}), says that if

\noindent\textit{
1)$p(x,D)$ has $C^\infty$ coefficients, and the lower order terms of $P(x,D)$ have bounded coefficients;\\
2)$S$ is $C^2$, and it is non-characteristic to $P$ at $x_0$, i.e. $p(x_0,\nabla\psi(x_0))\neq0$ and $\nabla\psi(x_0)\neq0$;\\
3)There exist a neighborhood $V$ of $x_0$ and a conic neighborhood $\Gamma$ of $\nabla\psi(x_0)$, such that when $(x,\xi,N,\tau)\in V\times\mathbb{R}^d\times\Gamma\times\mathbb{C}^d$ and $\xi\notin\mathbb{R}N$, the Fourier symbol of $p$ has the factorization
\begin{equation}\label{eq13}
p(x,\xi+\tau N)=p(x,N)\prod_k(\tau-a_k(x,\xi,N))^{\sigma_k}\prod_l(\tau-b_l(x,\xi,N))^{\tilde{\sigma}_l},
\end{equation}
where the distinguished $a_k$'s and $b_l$'s are respectively real and non-real valued $C^\infty$ functions in $ V\times\mathbb{R}^d\setminus\{0\}\times\Gamma$, and the integers $\sigma_k\in\{0,1\}$ and $\tilde{\sigma}_l\in\{0,1,2\}$,
}

\noindent then we can find a suitable weight function $\phi$ such that \eqref{eq14} can be proved with $\delta=1$. Moreover, if all $\tilde{\sigma}_l\leq1$ in \eqref{eq13}, then \eqref{eq14} can be proved with $\delta=\frac12$.

Notice that the principal part $p(x,D)$ and the above three conditions are invariant under any $C^2$ coordinates change, in the proof of Carleman estimate \eqref{eq14}, we may first assume by a local diffeomorphism that $\phi$ is some \textit{convex} quadratic form which is more computable, and then estimates for factors of the conjugated operator $p(\epsilon x,D+i\tau\nabla\phi)$ give \eqref{eq14} in the new coordinates, while real and non-real factors make different contribution:

\noindent1) Each non-real factor causes $\frac12$ loss of derivatives by a standard commutator type argument, where the convexity of $\phi$ is used. If any non-real factor has multiplicity $2$, the loss of derivatives will be piled up to $1$;\\
2) The real factors don't cause loss of derivatives but the loss of some ellipticity, and consequently the left hand side of \eqref{eq14} does not include terms of order $\mathrm{deg}\,P$. 

\noindent Since the diffeomorphism generically generates all lower order errors, the left hand side of \eqref{eq14} must have positive power in $\tau$ to absorb these errors in the original coordinates by assuming $\tau>0$ large, except the case $\delta=1$ where the errors of order $\mathrm{deg}\,P-1$ are controlled by taking $\epsilon$ small. 

We remark that to check the factorization \eqref{eq13} is not always easy, and H\"{o}rmander (see \cite{H4}) had developed the well celebrated pseudo convexity theory to study the case $\delta=\frac12$ by considering a more straightforward relation between $\phi$ and $p$. The more general case $\frac12\leq\delta<1$ has also been considered in Lerner \cite{L}, which takes spirit from the subelliptic operator theory where ellipticity is lost in a similar manner but lower order information can still be controlled. The bi-laplacian $\Delta^2$ is a typical operator for the critical case $\delta=1$, and there has been some technical variants in related problems, for example, Rousseau and Robbiano \cite{RR} recently considered the spectral inequality for the bi-Laplacian using Carleman estimates. One expects that the loss of derivatives $\delta$ will be higher than $1$ if in \eqref{eq13} the multiplicity of some non-real factor becomes higher, and that more loss of ellipticity happens if some higher order real factor occurs, thus the above strategy of proving a Carleman estimate fails in the end. One consults Zuily's book \cite{Z} for counterexamples and some special positive results for higher order operators. 

In our cases for $P_p(D_t,D_x)$ and $P_s(D_t,D_x)$, more concretely, the failure of the above classical ideas comes from two aspects. The first is that the principal symbol $p(\eta,\xi)=|\xi|^{2m}$ has dimensional loss, then for any $N_0\in\mathbb{R}$, $0\neq N\in\mathbb{R}^n$, one can always find $(\eta,\xi)\notin\mathbb{R}(N_0,N)$ such that $\tau\mapsto p(\eta+\tau N_0,\xi+\tau N)$ has a real root with multiplicity higher than $1$. The second aspect is the high power in the spatial part of these operators, for when $m>2$, possible non-real roots must have multiplicities higher than $2$. Notice that all these reasons have nothing to do yet with the hypersurface $S$, it seems that considering local unique continuation is somehow away from being prepared for us.

However, with a similar strategy of proving \eqref{eq14}, it is still possible to consider global unique continuation for $P_p(D_t,D_x)$ and $P_s(D_t,D_x)$ like Theorem \ref{thm11} and Theorem \ref{thm12}. We are first motivated by Isakov \cite{I} respecting the dimensional loss, where Isakov considered the non-homogeneous principal part of an operator in full dimensions, and gave an argument parallel to the pseudo convexity theory of H\"{o}rmander, with typical applications to second order evolution operators and one dimensional higher order evolution operators as mentioned before. In our case, $\partial_t$ and $D_t$ should not be neglected in the non-homogeneous principal parts of $P_p(D_t,D_x)$ and $P_s(D_t,D_x)$, however, we will not emphasize such concept later but study them in a more straightforward way. The other motivation of this work is technical respecting the high power. As analyzed above, since the loss of derivatives is expected to be high, generic lower order errors that are uncontrollable must be avoided. Similar consideration for higher order elliptic operators has been more studied, for example in \cite{P}, \cite{CG} and \cite{CK}, strong unique continuation properties were considered, and iteration of Carleman estimates for second order operators was used to avoid certain lower order errors. In this work, we will not consider any generic diffeomorphic scenario, but only quadratic forms as the weight function in a Carleman estimate, and take advance of the Tr\`{e}ves' identity which produces lower terms in a completely predictable way. By considering parameterized quadratic forms, we inspect what kinds of Carleman estimate can be proved and consequently what kinds of hypersurface will allow unique continuation to cross. It turns out that quadratic forms that only depend on $t$ and $x_n$ are suitable choices even if they are not convex, and therefore the spatial boundary in Theorem \ref{thm11} and Theorem \ref{thm12} can then only be local in $t$ and $x_n$.

Such global results can be expected to serve some further global unique continuation problems. For example, Kenig, Ponce and Vega \cite{KPV} considered unique continuation for the nonlinear Schr\"{o}dinger equation
\begin{equation*}
i\partial_tu+\Delta_xu+F(u,\bar{u})=0,\quad(t,x)\in\mathbb{R}\times\mathbb{R}^n,
\end{equation*} 
and proved that if two solutions coincide in $\{0,1\}\times D$ where the cone $D$ is strictly larger than a half space, then the two solutions are identically equal. They used the result of Isakov \cite{I} for Schr\"{o}dinger equations as a second step after showing that the two solutions coincide in $[0,1]\times D'$ where the cone $D'\subset D$ is also strictly larger than a half space. Compared with Isakov's result, our Theorem \ref{thm12} shows that $D$ can also be assumed to be a half space with the same regularity assumptions in \cite{KPV}. We also mention that Ionescu and Kenig \cite{IK} later developed the $L^p$ Carleman estimates which are also global but much harder to prove, to reduce $D$ to a half space, and we refer to \cite{EKPV} for more related problems for evolution equations. We expect our Theorem \ref{thm12} to be a step to such problems in the higher order case.

As a byproduct, the Carleman estimates (see Lemma \ref{lm31} and \ref{lm32}) that we use to prove Theorem \ref{thm11} and Theorem \ref{thm12} can be slightly modified to show a local unique continuation property for $P_p(D_t,D_x)$ and $P_s(D_t, D_x)$ (see Proposition \ref{prop51}), where the initial vanishing of the solutions is assumed to be on a saddle shape set. Besides being a new local unique continuation result in the higher order case in higher dimensions, this is also a "better" result than Isakov's in \cite{I} for the Schr\"{o}dinger equations, because the saddle shape set we consider is actually a smaller initial vanishing set, while the only one flaw is that we cannot include the full gradient in the equation.

This paper is organized as follows. In Section \ref{sec2} we introduce the Tr\`{e}ves' identity, upon the use of which we also prove some error estimates. In Section \ref{sec3} we prove two Carleman estimates which are the main gadgets to prove Theorem \ref{thm11} and Theorem \ref{thm12} in Section \ref{sec4}. In Section \ref{sec5} we prove a local unique continuation result by modifying the Carleman estimates in Section \ref{sec3}, and we also give a weak unique continuation result as a corollary. 

In the sequel, we will use the notation $D=i^{-1}\nabla=i^{-1}\partial$ indicating variables by subscript. For $d$ dimensional vector $b$ and multi-index $\alpha$, we denote $b^\alpha=b_1^{\alpha_1}\cdots b_d^{\alpha_d}$. For a polynomial $P$ in $\mathbb{C}^d$ with constant coefficients, we denote $\bar{P}$ the polynomial by taking conjugation of all coefficients of $P$, and $P^{(\alpha)}$ the mixed derivative $\partial^\alpha P$.

%%%%%%%%%%%%%%%%%%%%%%%%%%%%%%%%%%%%%%%%%%%%%%%%%%%%%%%%%%%%%%%%%%%%
\section{Preliminaries}\label{sec2}

The following Tr\`{e}ves' identity (see \cite[Lemma 17.2.2]{H3}), which explains the canonical commutation relation for general linear differential operators with constant coefficients, is our main tool to prove Carleman estimates:

\begin{lemma}\label{lm21}
	Let $Q(x)=\sum_{j=1}^{d}a_jx_j+\sum_{j=1}^{d}b_jx_j^2/2$ be a real quadratic polynomial in $\mathbb{R}^d$, $P$ be a $d$-dimensional polynomial with constant coefficients. Then for all $u\in C_c^\infty(\mathbb{R}^n)$ and denoted by $v=e^{Q/2}u$, we have
	\begin{equation}\label{eq21}
	\begin{split}
	\int e^Q|P(D)u|^2dx&=\int|P(D+i\nabla Q/2)v|^2dx\\
	&=\sum_{\alpha}\frac{b^\alpha}{\alpha!}\int|\bar{P}^{(\alpha)}(D-i\nabla Q/2)v|^2dx,
	\end{split}
	\end{equation}
	where the finite summation runs over all possible multi-indices $\alpha$.
\end{lemma}

We also need a further conclusion of the Tr\`{e}ves' identity:

\begin{lemma}\label{lm22}
If, in addition to Lemma \ref{lm21}, we have $b_j\geq0$ for $1\leq j\leq d$, then for any non-negative integer $k$, there exists $C=C(\mathrm{deg}\,P,k,d)>0$ such that
	\begin{equation}\label{eq22}
	\sum_{|\alpha|\geq k}b^\alpha\int|\bar{P}^{(\alpha)}(D-i\nabla Q/2)v|^2dx\geq C\sum_{|\alpha|\geq k}b^\alpha\int|P^{(\alpha)}(D+i\nabla Q/2)v|^2dx.
	\end{equation}
\end{lemma}

\begin{proof}
	Notice that when $0<\epsilon<1$ we have
	\begin{equation*}
	\begin{split}
	\sum_{|\alpha|\geq k}b^\alpha\int|\bar{P}^{(\alpha)}(D-i\nabla Q/2)v|^2dx&=\sum_{|\alpha|\geq k}b^\alpha\int e^{-Q}|\bar{P}^{(\alpha)}(D)(e^{\frac Q2}v)|^2dx\\
	\quad\geq&\sum_{|\alpha|\geq k}\epsilon^{\mathrm{deg}\,P-|\alpha|}b^\alpha\int e^{-Q}|\bar{P}^{(\alpha)}(D)(e^{\frac Q2}v)|^2dx,
	\end{split}
	\end{equation*}
	we can then apply Lemma \ref{lm21} to each term in the last line above to obtain
	\begin{equation*}
	\begin{split}
	&\sum_{|\alpha|\geq k}b^\alpha\int|\bar{P}^{(\alpha)}(D-i\nabla Q/2)v|^2dx\\
	\geq&\sum_{|\alpha|\geq k}\sum_{\beta}\frac{\epsilon^{\mathrm{deg}\,P-|\alpha|}b^\alpha(-b)^\beta}{\beta!}\int|P^{(\alpha+\beta)}(D+i\nabla Q/2)v|^2dx\\
	\geq&\sum_{|\alpha|\geq k}\epsilon^{\mathrm{deg}\,P-|\alpha|}b^\alpha\int|P^{(\alpha)}(D+i\nabla Q/2)v|^2dx\\
	&\quad\quad-\sum_{|\alpha|\geq k}\sum_{|\beta|\neq0}\frac{\epsilon^{\mathrm{deg}\,P-|\alpha|}b^{\alpha+\beta}}{\beta!}\int|P^{(\alpha+\beta)}(D+i\nabla Q/2)v|^2dx\\
	\geq&\sum_{|\alpha|\geq k}\epsilon^{\mathrm{deg}\,P-|\alpha|}b^\alpha\int|P^{(\alpha)}(D+i\nabla Q/2)v|^2dx\\
	&\quad\quad-\sum_{|\alpha|\geq k}\sum_{|\beta|\neq0}\epsilon^{\mathrm{deg}\,P-|\alpha+\beta|+1}b^{\alpha+\beta}\int|P^{(\alpha+\beta)}(D+i\nabla Q/2)v|^2dx\\
	=&\sum_{|\alpha|\geq k}\epsilon^{\mathrm{deg}\,P-|\alpha|}(1-\epsilon C_{\mathrm{deg}\,P,\alpha,d})b^\alpha\int|P^{(\alpha)}(D+i\nabla Q/2)v|^2dx.
	\end{split}
	\end{equation*}
	Then choosing $\epsilon$ small completes the proof.
\end{proof}

The following lemma for error estimates is somehow standard, but we still give the proof here to make a self-contained discussion.

\begin{lemma}\label{lm23}
	For fixed $K\in\mathbb{N}_+$ and $\delta_0,\tau_0>0$, there exists $C>0$ such that for all $u(s)\in C^\infty(\mathbb{R})$ we have
	\begin{equation}\label{eq23}
	|(D_s\pm i\tau(1+s))^Ku-D_s^Ku|\leq C\sum_{k<K}\tau^{K-k}|D_s^ku|,\quad|s|<\delta_0,~\tau>\tau_0.
	\end{equation}
	We also have
	\begin{equation}\label{eq24}
	|(D_s\pm i\tau(1+s))^Ku-(D_s\pm i\tau)^Ku|\leq C(1+\delta\tau)\sum_{k<K}\tau^{K-1-k}|D_s^ku|,\quad|s|<\delta\leq\delta_0,~\tau>\tau_0.
	\end{equation}
\end{lemma}

\begin{proof}
	We only prove the "+" case since the "-" case is parallel. The lemma is obvious when $K\leq2$. When $K>2$, for \eqref{eq23},
	\begin{equation*}
	\begin{split}
	(D_s+i\tau(1+s))^K=&D_s(D_s+i\tau(1+s))^{K-1}+i\tau(1+s)(D_s+i\tau(1+s))^{K-1}\\
	=&(D_s+i\tau(1+s))^{K-1}D_s+(K-1)\tau(D_s+i\tau(1+s))^{K-2}\\
	&\quad+i\tau(1+s)(D_s+i\tau(1+s))^{K-1},
	\end{split}
	\end{equation*}
	thus
	\begin{equation*}
	\begin{split}
	&|(D_s\pm i\tau(1+s))^Ku-D_s^Ku|\\
	\leq&|((D_s+i\tau(1+s))^{K-1}-D_s^{K-1})D_su|+(K-1)\tau|(D_s+i\tau(1+s))^{K-2}u|\\
	&\quad+(1+\delta_0)\tau|(D_s+i\tau(1+s))^{K-1}u|\\
	\leq&C\sum_{k<K}\tau^{K-k}|D_s^ku|+C\tau^{-1}\sum_{k<K-1}\tau^{K-k}|D_s^ku|\\
	\leq&C\sum_{k<K}\tau^{K-k}|D_s^ku|.
	\end{split}
	\end{equation*}
	For \eqref{eq24} similarly,
	\begin{equation*}
	\begin{split}
	&(D_s+i\tau(1+s))^K\\
	%=&(D_s+i\tau)(D_s+i\tau(1+s))^{K-1}+i\tau s(D_s+i\tau(1+s))^{K-1}\\
	=&(D_s+i\tau(1+s))^{K-1}(D_s+i\tau)+(K-1)\tau(D_s+i\tau(1+s))^{K-2}\\
	&\quad+i\tau s(D_s+i\tau(1+s))^{K-1},
	\end{split}
	\end{equation*}
	thus
	\begin{equation*}
	\begin{split}
	&|(D_s\pm i\tau(1+s))^Ku-(D_s+i\tau)^Ku|\\
	\leq&|((D_s+i\tau(1+s))^{K-1}-(D_s+i\tau)^{K-1})(D_s+i\tau)u|\\
	&\quad+(K-1)\tau|(D_s+i\tau(1+s))^{K-2}u|+\delta\tau|(D_s+i\tau(1+s))^{K-1}u|\\
	\leq&C(1+\delta+\tau^{-1})(1+\delta\tau)\sum_{k<K}\tau^{K-1-k}|D_s^ku|\\
	\leq&C(1+\delta\tau)\sum_{k<K}\tau^{K-1-k}|D_s^ku|.
	\end{split}
	\end{equation*}
\end{proof}
%%%%%%%%%%%%%%%%%%%%%%%%%%%%%%%%%%%%%%%%%%%%%%%%%%%%%%%
\section{Carleman estimates}\label{sec3}

In this section, we shall prove two Carleman estimates for the main theorems respectively. The first one is for the higher order parabolic case:

\begin{lemma}\label{lm31}
	Let $\phi=-N\frac{t^2}{2}+x_n+\frac{x_n^2}{2}$ where $N\geq0$. Then there exist $\tau_0=\tau_0(N)>0$, $C=C(\tau_0)>0$, and $0<\delta<1$ which is independent of $N$, such that when $\tau>\tau_0$ we have
	\begin{equation}\label{eq31}
	\begin{split}
	\tau^{-m}\iint e^{2\tau\phi}|D_tu|^2dxdt+\sum_{|\alpha|\leq2m}\tau^{2(\frac{3m}{2}-|\alpha|)}\iint e^{2\tau\phi}|D_x^\alpha u|^2dxdt\\
	\leq C\iint e^{2\tau\phi}|P_p(D_t,D_x)u|^2dxdt,\quad u\in C_c^\infty(U),
	\end{split}
	\end{equation}
	where $U=\{(t,x)\in\mathbb{R}^{1+n};~|t|<\delta,|x_n|<\delta\}$.
\end{lemma}

\begin{proof}
	We first apply Lemma \ref{lm21} to $P_p(\eta,\xi)=i\eta+P_2^m(\xi)$ with $Q=2\tau\phi$ where $P_2(\xi)=\sum_{j=1}^n\xi_j^2$. Denoted by $v=e^{\tau\phi}u$, we obtain for all $u\in C_c^\infty(U)$ that
	\begin{equation*}
	\begin{split}
	&\iint e^{2\tau\phi}|P_p(D_t,D_x)u|^2dxdt\\
	=&\iint|-i(D_t-i\tau\partial_t\phi)v+P_2^m(D_x-i\tau\nabla_x\phi)v|^2dxdt-2N\tau\iint|v|^2dxdt\\
	&+\sum_{k>0}\frac{(2\tau)^k}{k!}\iint|(\partial_n^kP_2^m)(D_x-i\tau\nabla_x\phi)v|^2dxdt\\
	\geq&\iint|-i(D_t-i\tau\partial_t\phi)v+P_2^m(D_x-i\tau\nabla_x\phi)v|^2dxdt\\
	&+2\tau\iint|(\partial_nP_2^m)(D_x-i\tau\nabla_x\phi)v|^2dxdt\\
	&+\frac{(2\tau)^m}{m!}\iint|(\partial_n^mP_2^m)(D_x-i\tau\nabla_x\phi)v|^2dxdt\\
	&+\left(\frac{(2\tau)^{2m}}{(2m)!}(\partial_n^{2m}P_2^m)^2-2N\tau\right)\iint|v|^2dxdt,
	\end{split}
	\end{equation*}
	and here $\partial_n^{2m}P_2^m$ is a constant. When $\tau\geq\frac12\left(\frac{2N(2m)!}{(\partial_n^{2m}P_2^m)^2}\right)^\frac{1}{2m-1}:=\tau_1$, we have
	\begin{equation}\label{eq32}
	\begin{split}
	&C\tau^m\iint e^{2\tau\phi}|P_p(D_t,D_x)u|^2dxdt\\
	\geq&\iint|-i(D_t-i\tau\partial_t\phi)v+P_2^m(D_x-i\tau\nabla_x\phi)v|^2dxdt\\
	&+\tau^2\iint|(\partial_nP_2^m)(D_x-i\tau\nabla_x\phi)v|^2dxdt\\
	&+\tau^{2m}\iint|(\partial_n^mP_2^m)(D_x-i\tau\nabla_x\phi)v|^2dxdt.
	\end{split}
	\end{equation}
	Notice that $\partial_t\phi=-Nt$, $\nabla_x\phi=(0,\cdots,1+x_n)$, $\mathrm{deg}(\partial_nP_2^m)=2m-1$ and $\mathrm{deg}(\partial_n^mP_2^m)=m$, for any $0<\delta<1$ and $\tau>\tau_1$, we can use \eqref{eq24} to treat each term in \eqref{eq32} in the following ways:
	\begin{equation}\label{eq33}
	\begin{split}
	&\iint|-i(D_t-i\tau\partial_t\phi)v+P_2^m(D_x-i\tau\nabla_x\phi)v|^2dxdt\\
	\geq&\frac12\iint|-iD_tv+P_2^m(D_x-i\tau e_n)v|^2dxdt-CN^2\tau^2\iint|v|^2dxdt\\
	&\quad-C(1+\delta\tau)^2\sum_{|\alpha|<2m}\tau^{2(2m-1-|\alpha|)}\iint|D_x^\alpha v|^2dxdt,
	\end{split}
	\end{equation}
	\begin{equation}\label{eq34}
	\begin{split}
	&\iint|(\partial_nP_2^m)(D_x-i\tau\nabla_x\phi)v|^2dxdt\\
	\geq&\frac12\iint|(\partial_nP_2^m)(D_x-i\tau e_n)v|^2dxdt\\
	&\quad-C(1+\delta\tau)^2\sum_{|\alpha|<2m-1}\tau^{2(2m-2-|\alpha|)}\iint|D_x^\alpha v|^2dxdt,
	\end{split}
	\end{equation}
	and
	\begin{equation}\label{eq35}
	\begin{split}
	&\iint|(\partial_n^mP_2^m)(D_x-i\tau\nabla_x\phi)v|^2dxdt\\
	\geq&\frac12\iint|(\partial_n^mP_2^m)(D_x-i\tau e_n)v|^2dxdt\\
	&\quad-C(1+\delta\tau)^2\sum_{|\alpha|<m}\tau^{2(m-1-|\alpha|)}\iint|D_x^\alpha v|^2dxdt.
	\end{split}
	\end{equation}
	Combining \eqref{eq32}-\eqref{eq35} we have
	\begin{equation}\label{eq36}
	\begin{split}
	&C\tau^m\iint e^{2\tau\phi}|P_p(D_t,D_x)u|^2dxdt\\
	\geq&\iint|-iD_tv+P_2^m(D_x-i\tau e_n)v|^2dxdt\\
	&+\tau^2\iint|(\partial_nP_2^m)(D_x-i\tau e_n)v|^2dxdt\\
	&+\tau^{2m}\iint|(\partial_n^mP_2^m)(D_x-i\tau e_n)v|^2dxdt\\
	&-C(1+\delta\tau)^2\sum_{|\alpha|<2m}\tau^{2(2m-1-|\alpha|)}\iint|D_x^\alpha v|^2dxdt-CN^2\tau^2\iint|v|^2dxdt.
	\end{split}
	\end{equation}	
	
	We next claim that
	\begin{equation}\label{eq37}
	\begin{split}
	&\iint|-iD_tv+P_2^m(D_x-i\tau e_n)v|^2dxdt+\tau^2\iint|(\partial_nP_2^m)(D_x-i\tau e_n)v|^2dxdt\\
	&\quad+\tau^{2m}\iint|(\partial_n^mP_2^m)(D_x-i\tau e_n)v|^2dxdt\\
	&\geq C\left(\iint|D_tv|^2dxdt+\sum_{|\alpha|\leq2m}\tau^{2(2m-|\alpha|)}\iint|D_x^\alpha v|^2dxdt\right).
	\end{split}
	\end{equation}
	By Parseval's formula, it suffices to show 
	\begin{equation}\label{eq38}
	\begin{split}
	&|-i\eta+P_2^m(\xi-i\tau e_n)|^2+\tau^2|(\partial_nP_2^m)(\xi-i\tau e_n)|^2+\tau^{2m}|(\partial_n^mP_2^m)(\xi-i\tau e_n)|^2\\
	&\geq C(|\eta|^2+|\xi-i\tau e_n|^{4m})
	\end{split}
	\end{equation}
	for all $(\tau,\eta,\xi)\in\mathbb{R}^{1+1+n}$. With the bijection $\eta=|\widetilde{\eta}|^{2m-1}\widetilde{\eta}$ in $\mathbb{R}$, it is equivalent to show for all $(\tau,\widetilde{\eta},\xi)\in\mathbb{R}^{1+1+n}$ that
	\begin{equation}\label{eq39}
	\begin{split}
	&|-i|\widetilde{\eta}|^{2m-1}\widetilde{\eta}+P_2^m(\xi-i\tau e_n)|^2+\tau^2|(\partial_nP_2^m)(\xi-i\tau e_n)|^2+\tau^{2m}|(\partial_n^mP_2^m)(\xi-i\tau e_n)|^2\\
	\geq&C(|\widetilde{\eta}|^{4m}+|\xi-i\tau e_n|^{4m}).
	\end{split}
	\end{equation}
	Since both sides of \eqref{eq39} are homogeneous in $(\tau,\widetilde{\eta},\xi)\in\mathbb{R}^{1+1+n}$ of degree $4m$, and the right hand side is elliptic, we must show that the left hand side is also elliptic, which is a straight consequence of the following facts:
	\begin{equation}\label{eq310}
	\begin{cases}
	(\partial_nP_2^m)(\xi-i\tau e_n)=mP_2^{m-1}(\xi-i\tau e_n)(\partial_nP_2)(\xi-i\tau e_n),\\
	(\partial_n^mP_2^m)(\xi-i\tau e_n)=m!((\partial_nP_2)(\xi-i\tau e_n))^m+P_2(\xi-i\tau e_n)\widetilde{P}(\xi-i\tau e_n),\\
	(\partial_nP_2)(\xi-i\tau e_n)=2(\xi_n-i\tau),
	\end{cases}
	\end{equation}
	where $\widetilde{P}$ is some polynomial.
	
	Now Compare the last term in \eqref{eq36} and the last term when $\alpha=0$ in \eqref{eq37}, if $\tau\gg N^\frac{1}{2m-1}:=\tau_2$, we can combine \eqref{eq36} and \eqref{eq37} to have for some constants $C_1$ and $C_2$ that
	\begin{equation}\label{eq311}
	\begin{split}
	&\tau^m\iint e^{2\tau\phi}|P_p(D_t,D_x)u|^2dxdt\\
	\geq& C_1\left(\iint|D_tv|^2dxdt+\sum_{|\alpha|\leq2m}\tau^{2(2m-|\alpha|)}\iint|D_x^\alpha v|^2dxdt\right)\\
	&-C_2(1+\delta\tau)^2\sum_{|\alpha|<2m}\tau^{2(2m-1-|\alpha|)}\iint|D_x^\alpha v|^2dxdt.
	\end{split}
	\end{equation}
	If we choose $\delta\leq\sqrt{\frac{C_1}{8C_2}}$ and $\tau\geq\sqrt{\frac{8C_2}{C_1}}:=\tau_3$, then $C_2(1+\delta\tau)^2\leq\frac{C_1}{2}\tau^2$, and therefore
	\begin{equation}\label{eq312}
	\begin{split}
	&C\tau^m\iint e^{2\tau\phi}|P_p(D_t,D_x)u|^2dxdt\\
	\geq& \iint|D_tv|^2dxdt+\sum_{|\alpha|\leq2m}\tau^{2(2m-|\alpha|)}\iint|D_x^\alpha v|^2dxdt.
	\end{split}
	\end{equation}
	Notice that $D_tv=e^{\tau\phi}(D_t+i\tau Nt)u$, we have
	\begin{equation}\label{eq313}
	\iint|D_tv|^2dxdt\geq\frac12\iint e^{2\tau\phi}|D_tu|^2dxdt-CN^2\tau^2\iint e^{2\tau\phi}|u|^2dxdt.
	\end{equation}
	Also notice that $D_x^\alpha v=e^{\tau\phi}(D_x-i\tau(1+x_n)e_n)^\alpha u$, then for $0<\epsilon<1$ small enough, we have by \eqref{eq23} that 
	\begin{equation}\label{eq314}
	\begin{split}
	&\sum_{|\alpha|\leq2m}\tau^{-2|\alpha|}\iint|D_x^\alpha v|^2dxdt\\
	\geq&\sum_{|\alpha|\leq2m}\epsilon^{|\alpha|}\tau^{-2|\alpha|}\iint e^{2\tau\phi}|(D_x-i\tau(1+x_n)e_n)^\alpha u|^2dxdt\\
	\geq&\frac12\sum_{|\alpha|\leq2m}\epsilon^{|\alpha|}\tau^{-2|\alpha|}\iint e^{2\tau\phi}|D_x^\alpha u|^2dxdt\\
	&\quad\quad\quad-C\sum_{0<|\alpha|\leq2m}\sum_{\beta<\alpha}\epsilon^{|\alpha|}\tau^{-2|\beta|}\iint e^{2\tau\phi}|D_x^\beta u|^2dxdt\\
	\geq&\frac12\sum_{|\alpha|\leq2m}\epsilon^{|\alpha|}\tau^{-2|\alpha|}\iint e^{2\tau\phi}|D_x^\alpha u|^2dxdt\\
	&\quad\quad\quad-C'\sum_{|\alpha|<2m}\epsilon^{|\alpha|+1}\tau^{-2|\alpha|}\iint e^{2\tau\phi}|D_x^\alpha u|^2dxdt\\
	\geq&C\sum_{|\alpha|\leq2m}\tau^{-2|\alpha|}\iint e^{2\tau\phi}|D_x^\alpha u|^2dxdt.
	\end{split}
	\end{equation}
	Finally, since $\tau\gg\tau_2$ again, \eqref{eq312}-\eqref{eq314} prove \eqref{eq31} with $\tau_0\gg\mathrm{max}\{\tau_1,\tau_2,\tau_3\}$.
\end{proof}

\begin{remark}\label{rk31}
	Since $\phi$ is independent of $x'=(x_1,\cdots,x_{n-1})$, Lemma \ref{lm31} actually holds for $u(t,x',x_n)$ with $\mathrm{supp}\,u\subset U(\delta')=\{(t,x)\in\mathbb{R}^{1+n};~|t|<\delta',|x_n|<\delta'\}$ and $\partial_tu,\partial_x^\alpha u\in L^2(U(\delta'))$, $|\alpha|\leq2m$, for some $\delta'<\delta$. This also applies to the next Lemma \ref{lm32}.
\end{remark}

We next consider the higher order Schr\"{o}dinger case, where the Carleman estimate obtained will be slightly weaker than \eqref{eq31} formally, due to some lack of anisotropic ellipticity respecting \eqref{eq38} and \eqref{eq39}.

\begin{lemma}\label{lm32}
	Let $\phi=-N\frac{t^2}{2}+x_n+\frac{x_n^2}{2}$ where $N\geq0$. Then there exist $\tau_0=\tau_0(N)>0$, $C=C(\tau_0)>0$, and $0<\delta<\frac12$ which is independent of $N$, such that when $\tau>\tau_0$ we have
	\begin{equation}\label{eq315}
	\begin{split}
	\sum_{|\alpha|\leq2m-2}\tau^{2(\frac{3m}{2}-|\alpha|)}\iint e^{2\tau\phi}|D_x^\alpha u|^2dxdt+\sum_{|\alpha|=2m-2}\tau^{2(1-\frac m2)}\iint e^{2\tau\phi}|D_x^{\alpha+e_n}u|^2dxdt\\
	\leq C\iint e^{2\tau\phi}|P_s(D_t,D_x)u|^2dxdt,\quad u\in C_c^\infty(U),
	\end{split}
	\end{equation}
	where $U=\{(t,x)\in\mathbb{R}^{1+n};~|t|<\delta,|x_n|<\delta\}$.
\end{lemma}

\begin{proof}
	As in the proof of Lemma \ref{lm31}, we first apply Lemma \ref{lm21} to $P_s(\eta,\xi)=\eta+P_2^m(\xi)$ with $Q=2\tau\phi$ where $P_2(\xi)=\sum_{j=1}^n\xi_j^2$, and obtain for all $u\in C_c^\infty(U)$ that
	\begin{equation*}
	\begin{split}
	&\iint e^{2\tau\phi}|P_s(D_t,D_x)u|^2dxdt\\
	=&\iint|(D_t-i\tau\partial_t\phi)v+P_2^m(D_x-i\tau\nabla_x\phi)v|^2dxdt-2N\tau\iint|v|^2dxdt\\
	&+\sum_{k>0}\frac{(2\tau)^k}{k!}\iint|(\partial_n^kP_2^m)(D_x-i\tau\nabla_x\phi)v|^2dxdt,
	\end{split}
	\end{equation*}
	where $v=e^{\tau\phi}u$. Discarding the first term on the right hand side and apply \eqref{eq22} to the last sum, we have
	\begin{equation}\label{eq316}
	\begin{split}
	&\iint e^{2\tau\phi}|P_s(D_t,D_x)u|^2dxdt\\
	\geq&C\tau\iint|(\partial_nP_2^m)(D_x+i\tau\nabla_x\phi)v|^2dxdt+\left(C'\tau^{2m}-2N\tau\right)\iint|v|^2dxdt.
	\end{split}
	\end{equation}
	When $\tau\geq(\frac{2N}{C'}):=\tau_1$, the last term on the right hand side of \eqref{eq316} is discarded. For the remaining term, observe that
	\begin{equation*}
	\begin{split}
	(\partial_nP_2^m)(D_x+i\tau\nabla_x\phi)&=2m(D_n+i\tau\partial_n\phi)P_2^{m-1}(D_x+i\tau\nabla_x\phi)\\
	&=2mP_2^{m-1}(D_x+i\tau\nabla_x\phi)(D_n+i\tau\partial_n\phi),
	\end{split}
	\end{equation*}
	denoted by $v_n=(D_n+i\tau\partial_n\phi)v$, we then have
	\begin{equation}\label{eq317}
	\begin{split}
	&C\iint e^{2\tau\phi}|P_s(D_t,D_x)u|^2dxdt\\
	\geq&\tau\iint|P_2^{m-1}(D_x+i\tau\nabla_x\phi)v_n|^2dxdt\\
	&\quad+\tau\iint|(D_n+i\tau\partial_n\phi)P_2^{m-1}(D_x+i\tau\nabla_x\phi)v|^2dxdt.
	\end{split}
	\end{equation}
	Also notice that the operator
	\begin{equation*}
	\begin{split}
	(D_n-i\tau\partial_n\phi)(D_n+i\tau\partial_n\phi)&=D_n^2+\tau\partial_n^2\phi+\tau^2(\partial_n\phi)^2\\
	&=D_n^2+\tau+\tau^2(1+x_n)^2,
	\end{split}
	\end{equation*}
	we integrate by parts and use the fact that $(1+x_n)^2\geq(1-\delta)^2\geq\frac14$ in $U$, to obtain
	\begin{equation}\label{eq318}
	\begin{split}
	&C\iint e^{2\tau\phi}|P_s(D_t,D_x)u|^2dxdt\\
	\geq&\tau\iint|P_2^{m-1}(D_x+i\tau\nabla_x\phi)v_n|^2dxdt+\tau^3\iint|P_2^{m-1}(D_x+i\tau\nabla_x\phi)v|^2dxdt.
	\end{split}
	\end{equation}
	
	We only outline how to estimate the second term on the right hand side of \eqref{eq318}, since it is similar to the proof of Lemma \ref{lm31} and it also applies to the first term with no change. We first use Lemma \ref{lm21} to show that
	\begin{equation*}
	\begin{split}
	&C\tau^{m-1}\iint|P_2^{m-1}(D_x+i\tau\nabla_x\phi)v|^2dxdt\\
	\geq&\iint|P_2^{m-1}(D_x-i\tau\nabla_x\phi)v|^2dxdt+\tau^2\iint|(\partial_nP_2^{m-1})(D_x-i\tau\nabla_x\phi)v|^2dxdt\\
	&\quad+\tau^{2m-2}\iint|(\partial_n^{m-1}P_2^{m-1})(D_x-i\tau\nabla_x\phi)v|^2dxdt.
	\end{split}
	\end{equation*}
	An argument completely parallel to \eqref{eq33}-\eqref{eq312} shows that for small $\delta$ and large $\tau$ we have
	\begin{equation*}
	\iint|P_2^{m-1}(D_x+i\tau\nabla_x\phi)v|^2dxdt\geq C\sum_{|\alpha|\leq2m-2}\tau^{2(\frac{3m}{2}-\frac32-|\alpha|)}\iint|D_x^\alpha v|^2dxdt.
	\end{equation*}
	Plug this back in \eqref{eq318}, we get
	\begin{equation*}
	\begin{split}
	&C\iint e^{2\tau\phi}|P_s(D_t,D_x)u|^2dxdt\\
	\geq&\sum_{|\alpha|\leq2m-2}\tau^{2(\frac{3m}{2}-1-|\alpha|)}\iint|D_x^\alpha v_n|^2dxdt+\sum_{|\alpha|\leq2m-2}\tau^{2(\frac{3m}{2}-|\alpha|)}\iint|D_x^\alpha v|^2dxdt.
	\end{split}
	\end{equation*}
	Finally, because $D_x^\alpha v_n=e^{\tau\phi}(D_x-i\tau(1+x_n)e_n)^\alpha D_nu$ and $D_x^\alpha v=e^{\tau\phi}(D_x-i\tau(1+x_n)e_n)^\alpha u$, an argument which is the same to \eqref{eq314} completes the proof.
\end{proof}

%%%%%%%%%%%%%%%%%%%%%%%%%%%%%%%%%%%%%%%%%%%%%%%%%%%%%%%%%%%%%%

\section{Proofs of Theorem \ref{thm11} and Theorem \ref{thm12}}\label{sec4}

The two proofs are almost the same.
\begin{proof}
	We first prove Theorem \ref{thm11}. By translation, it is equivalent to prove that $u(t,x',x_n)\equiv0$ in some neighborhood of $\{0\}\times\mathbb{R}^{n-1}\times(-B',0)$, and $B'>0$ can be chosen independent of $A$. Let $\epsilon\in(0,1)$ be chosen later but fixed, and set $u_\epsilon(t,x)=u(\epsilon^{2m}t,\epsilon x)$. Then by \eqref{eq11},
	\begin{equation}\label{eq41}
	|P_p(D_t,D_x)u_\epsilon|\leq C\sum_{|\alpha|\leq[\frac{3m}{2}]}\epsilon^{2m-|\alpha|}|\partial_x^\alpha u_\epsilon|
	\end{equation}
	holds in $(-\epsilon^{-2m}A,\epsilon^{-2m}A)\times\mathbb{R}^{n-1}\times(-\epsilon^{-1}B,\epsilon^{-1}B)$. We first choose $\delta\in(0,B)$ such that Lemma \ref{lm31} is valid, and then choose $\delta_0(\epsilon,A,\delta)<\mathrm{min}\{\delta,\epsilon^{-2m}A\}$. Next we take $\theta\in C_c^\infty(-\delta_0,\delta_0)$ and $\chi\in C_c^\infty(-\delta,\delta)$, such that $\theta\equiv1$ in $(-\frac{\delta_0}{2},\frac{\delta_0}{2})$ and $\chi\equiv1$ in $(-\frac{\delta}{2},\frac{\delta}{2})$. Denoted by $U_\epsilon(t,x)=\theta(t)\chi(x_n)u_\epsilon(t,x)$, and notice that
	\begin{equation}\label{eq42}
	P_p(D_t,D_x)U_\epsilon=\theta\chi P_p(D_t,D_x)u_\epsilon+\theta'\chi u_\epsilon+\sum_{|\alpha|<2m}C_{\alpha}\theta\chi^{(2m-|\alpha|)}\partial_x^\alpha u_\epsilon,
	\end{equation}
	in the view of Fourier transform and by the regularity assumption on $u$, it follows that $P_p(D_t,D_x)U_\epsilon\in L^2(X)$ and thus $\partial_tU_\epsilon,\partial_x^\alpha U_\epsilon\in L^2(X)$ for $|\alpha|\leq2m$. Therefore by Remark \ref{rk31}, we can apply Lemma \ref{lm31} to $U_\epsilon$ with $N=4\delta_0^{-2}\delta(1-\frac\delta4)$ and obtain when $\tau>\tau_0$ that
	\begin{equation}\label{eq43}
	\begin{split}
	&C\sum_{|\alpha|\leq[\frac{3m}{2}]}\tau^{2(\frac{3m}{2}-|\alpha|)}\iint e^{2\tau\phi}|D_x^\alpha U_\epsilon|^2dxdt\\
	\leq&\sum_{|\alpha|\leq[\frac{3m}{2}]}\epsilon^{2(2m-|\alpha|)}\iint_{|t|<\frac{\delta_0}{2},|x_n|<\frac\delta2}e^{2\tau\phi}|D_x^\alpha U_\epsilon|^2dxdt\\
	&\quad+\left(\iint_{\frac{\delta_0}{2}<|t|<\delta_0,|x_n|<\delta}+\iint_{|t|<\delta_0,\frac{\delta}{2}<|x_n|<\delta}\right)e^{2\tau\phi}|P_p(D_t,D_x)U_\epsilon|^2dxdt,
	\end{split}
	\end{equation}
	while the first line on the right hand side comes from \eqref{eq41} and the fact that $U_\epsilon\equiv u_\epsilon$ when $|t|<\frac{\delta_0}{2}$ and $|x_n|<\frac\delta2$. When $\epsilon$ is chosen small and $\tau$ is large, the first line on the right hand side of \eqref{eq43} is absorbed into the left hand side. One checks that the second two terms are bounded by $C'e^{-\tau\delta(1-\frac\delta4)}$ due to the regions of integrations, thus we have
	\begin{equation}
	\tau^{3m}\iint e^{2\tau\phi}|U_\epsilon|^2dxdt\leq Ce^{-\tau\delta(1-\frac\delta4)}.
	\end{equation}
	Notice that for every $x_n\in(-\frac\delta2,0)$ we have $x_n+\frac{x_n^2}{2}>-\frac\delta2(1-\frac\delta4)$, then there exists $\widetilde{\delta_0}(x_n)\in(0,\frac{\delta_0}{2})$ such that if $-\frac\delta2<x_n<0$ and $|t|<\widetilde{\delta_0}(x_n)$, we have $\theta(t)\equiv\chi(x_n)\equiv1$,
	\begin{equation}
	\phi\geq-\frac N2\widetilde{\delta_0}^2+x_n+\frac{x_n^2}{2}\geq-\frac\delta2(1-\frac\delta4),
	\end{equation}
	and therefore
	\begin{equation}
	\tau^{3m}\iint_{|t|<\widetilde{\delta_0}(x_n),-\frac\delta2<x_n<0}|u_\epsilon|^2dxdt\leq C.
	\end{equation}
	Let $\tau\rightarrow+\infty$, we prove that $u_\epsilon\equiv0$ in a neighborhood of $\{0\}\times\mathbb{R}^{n-1}\times(-\frac\delta2,0)$, i.e. $u\equiv0$ in a neighborhood of $\{0\}\times\mathbb{R}^{n-1}\times(-\frac{\epsilon\delta}{2},0)$. Finally, recall that the choices of $\delta$ and $\epsilon$ only depend on Lemma \ref{lm31}, $B$ and the differential inequality \eqref{eq11}, thus we can take $B'=\frac{\epsilon\delta}{2}$ to complete the proof of Theorem \ref{thm11}.
	
	For the prove of Theorem \ref{thm12}, we use \eqref{eq12} and Lemma \ref{lm32} instead of \eqref{eq11} and Lemma \ref{lm31}. Then all the above details are correspondingly replaced for $P_s(D_t,D_x)$, except that \eqref{eq42} does not imply $\partial_tU_\epsilon,\partial_x^\alpha U_\epsilon\in L^2(X)$ for $|\alpha|\leq2m$, which we have assumed $\partial_tu,\partial_x^\alpha u\in L^2(X)$ for $|\alpha|\leq2m$ instead to ensure.
\end{proof}
%%%%%%%%%%%%%%%%%%%%%%%%%%%%%%%%%%%%%%%%%%%%%%%%%%%%%%%%%%%%%%%

\section{Local and weak unique continuation results}\label{sec5}

In the proofs of Theorem \ref{thm11} and Theorem \ref{thm12}, the zero level set of $\phi$ intersects the hyperplane on a line, but thanks to Remark \ref{rk31}, we don't have to cut off the solution in the $x'$ variables, and the essential intersection is a single point within the $(t,x_n)$-plane, which leads Carleman estimates to work. It is not surprising that a more standard use of Lemma \ref{lm31} and Lemma \ref{lm32} can also yield local results considering solutions which have initial vanishing strictly larger than the half space locally. Actually, vanishing on a saddle shape set is enough, and we will slightly modify Lemma \ref{lm31} and Lemma \ref{lm32} to prove the following local unique continuation property:

\begin{proposition}\label{prop51}
	Let $X$ be a neighborhood of $0\in\mathbb{R}^{1+n}$. Suppose $\partial_x^\alpha u\in L^2(X)$ for $|\alpha|<2m$, $P_p(D_t,D_x)u\in L^2(X)$ and \eqref{eq11} holds in $X$; or suppose $\partial_tu,\partial_x^\alpha u\in L^2(X)$ for $|\alpha|\leq2m$, and \eqref{eq12} holds in $X$. If $\mathrm{supp}\,u$ is contained in a saddle shape set
	\begin{equation}
	F=\left\{(t,x)\in X;~-\varphi(t)+x_n+\mbox{$N\frac{x_n^2}{2}$}+f(x')\leq0\right\}
	\end{equation}
	for some $C^2\ni\varphi\geq0$, $N>0$ and $f\geq0$, with the property that $\varphi(t)+f(x')=0\Rightarrow(t,x')=0$, then $u\equiv0$ in a neighborhood of $0$.
\end{proposition}

\begin{proof}
	If $\phi=-2\varphi(t)+x_n+N\frac{x_n^2}{2}$, we first outline for how \eqref{eq31} and \eqref{eq315} still hold with small $\delta$ and large $\tau$ by almost the same proofs of Lemma \ref{lm31} and \ref{lm32}. For the higher order parabolic case, notice that if $f\in C_c^\infty(X)$ and $g=e^{\tau\phi}f$, then
	\begin{equation}\label{eq51}
	\begin{split}
	&\iint e^{2\tau\phi}|P_p(D_t,D_x)f|^2dxdt\\
	=&\iint|i(D_t+i\tau\partial_t\phi)g|^2dxdt+2\mathrm{Re}\iint i(D_t+i\tau\partial_t\phi)g\overline{P_2^m(D_x+i\tau\nabla_x\phi)g}dxdt\\
	&\quad+\iint|P_2^m(D_x+i\tau\nabla_x\phi)g|^2dxdt.
	\end{split}
	\end{equation}
	Integrating by parts shows that
	\begin{equation*}
	\iint|i(D_t+i\tau\partial_t\phi)g|^2dxdt\geq\iint|-i(D_t-i\tau\partial_t\phi)g|^2dxdt-4\tau||\varphi''||_{L^\infty}\iint|g|^2dxdt.
	\end{equation*}
	Since $(\partial_t\phi,\nabla_x\phi)=(-2\varphi'(t),0,\cdots,1+Nx_n)$ and thus the commutator $[D_t+i\tau\partial_t\phi,P_2^m(D_x-i\tau\nabla_x\phi)]=0$, it is obvious from integration by parts that
	\begin{equation*}
	\begin{split}
	&2\mathrm{Re}\iint i(D_t+i\tau\partial_t\phi)g\overline{P_2^m(D_x+i\tau\nabla_x\phi)g}dxdt\\
	=&-2\mathrm{Re}\iint i(D_t-i\tau\partial_t\phi)g\overline{P_2^m(D_x-i\tau\nabla_x\phi)g}dxdt.
	\end{split}
	\end{equation*}
	Plug them back in \eqref{eq51} and apply Lemma \ref{lm21} to $\iint|P_2^m(D_x+i\tau\nabla_x\phi)g|^2dxdt$ in \eqref{eq51}, we then have
	\begin{equation*}
	\begin{split}
	&\iint e^{2\tau\phi}|P_p(D_t,D_x)f|^2dxdt\\
	\geq&\iint|-i(D_t-i\tau\partial_t\phi)g+P_2^m(D_x-i\tau\nabla_x\phi)g|^2dxdt-C\tau\iint|g|^2dxdt\\
	&+C(N)\sum_{k>0}\tau^k\iint|(\partial_n^kP_2^m)(D_x-i\tau\nabla_x\phi)g|^2dxdt.
	\end{split}
	\end{equation*}
	Therefore \eqref{eq31} holds for such $\phi$ if we follow the proof of Lemma \ref{lm31} from line to line. The higher order Schr\"{o}dinger case is completely parallel.
	
	Next, we show the uniqueness. For the higher order parabolic case, we set $u_\epsilon(t,x)=u(\epsilon^{2m}t,\epsilon x)$, then \eqref{eq41} holds in $\epsilon^{-1}X$. Take $\chi\in C_c^\infty(\epsilon^{-1}X)$ such that $\chi\equiv1$ near $0$ with $\mathrm{supp}\,\chi\subset B(0;\delta)$, and let $U_\epsilon=\chi u_\epsilon$. By our regularity assumption, as in Section \ref{sec4}, we can apply \eqref{eq31} to $U_\epsilon$ with $\phi$ that we set at the beginning of this proof:
	\begin{equation}\label{eq52}
	\begin{split}
	&C\sum_{|\alpha|\leq[\frac{3m}{2}]}\tau^{2(\frac{3m}{2}-|\alpha|)}\iint e^{2\tau\phi}|D_x^\alpha U_\epsilon|^2dxdt\\
	\leq&\sum_{|\alpha|\leq[\frac{3m}{2}]}\epsilon^{2(2m-|\alpha|)}\iint_{\chi\equiv1}e^{2\tau\phi}|D_x^\alpha U_\epsilon|^2dxdt\\
	&\quad\quad+\iint_{0<\chi<1}e^{2\tau\phi}|P_p(D_t,D_x)U_\epsilon|^2dxdt.
	\end{split}
	\end{equation}
	The first sum on the right hand side is absorbed if $\epsilon$ is small and $\tau$ is large. Now the support assumption implies that $\mathrm{supp}\,u\cap\{(t,x)\in X_0;~\phi(t,x)\geq0\}\subset\{0\}$ if $X_0\ni0$ is sufficiently small, therefore the last term on the right hand side of \eqref{eq52} is bounded by $Ce^{-2c\tau}$ for some $c>0$. Thus we have
	\begin{equation}
	\tau^{3m}\iint_{\phi\geq-c}|u_\epsilon|^2dxdt\leq C.
	\end{equation}
	Letting $\tau\rightarrow+\infty$ completes the proof, and the high order Schr\"{o}dinger case is parallel.
\end{proof}

We remark that in the second order case $m=1$ as mentioned in the Introduction, Isakov \cite{I} or the earlier Saut and Scheurer \cite{SS} have shown that $P_p(D_t,D_x)$ actually has local unique continuation property stronger than the above Proposition \ref{prop51} claims, which implies that if the solution vanishes in an open set in $\mathbb{R}^{1+n}$, then it vanishes in the horizontal (with respect to the time variable $t$) connected component of this open set where the equation is defined. \cite{SS} also claims such property for $P_p(D_t,D_x)$ in the higher order case, but it seems that the details were given in full. In the case $m=1$ for $P_s(D_t,D_x)$, the initial vanishing requirement of Proposition \ref{prop51} is weaker than the example for Schr\"{o}dinger equations shown in Isakov \cite{I} where, however, the full gradient is allowed in the lower order terms, technically due to that the anisotropic pseudo convexity theory still works somehow.

We finally mention a weak unique continuation result which comes straightforward from the main theorems:

\begin{corollary}
	Let $G$ be a bounded open subset of $\mathbb{R}^n$, and $X=(-T,T)\times G$. Suppose $\partial_x^\alpha u\in L^2(X)$ for $|\alpha|<2m$, $P_p(D_t,D_x)u\in L^2(X)$ and \eqref{eq11} holds in $X$; or suppose $\partial_tu,\partial_x^\alpha u\in L^2(X)$ for $|\alpha|\leq2m$, and \eqref{eq12} holds in $X$. If $\mathrm{supp}\,u\subset (-T,T)\times K$ where $K\Subset G$, then $u\equiv0$ in $X$.
\end{corollary}

\begin{proof}
	The regularity and support assumptions allow us to zero extend $u$ to $(-T,T)\times\mathbb{R}^n$, then the Corollary is implied by Theorem \ref{thm11} or Theorem \ref{thm12}.
\end{proof}

\noindent
\section*{Acknowledgements}
The author was visiting The University of Chicago (from Huazhong University of Science and Technology, supported by the China Scholarship Council,) while this research was carried out, and he thanks Carlos E. Kenig for fruitful discussion on this work. The author also thanks Quan Zheng for some helpful discussion on related topics, and the anonymous referee for the advice on the exposition of this manuscript.

\end{document}